\documentclass[12pt,reqno]{amsart}

\usepackage{amssymb,bbm,units,url}
\usepackage{booktabs}
\usepackage[bookmarks=false,unicode,colorlinks,urlcolor=red,citecolor=red,linkcolor=blue]{hyperref}
\usepackage{aliascnt}
\usepackage{doi}
\usepackage{tikz}
\usepackage{pgfplots} 
\usepackage{color}

\usepackage[margin=1.25in]{geometry}

\newtheorem{theorem}{Theorem}%[section]

\newaliascnt{lemma}{theorem}
\newtheorem{lemma}[lemma]{Lemma}
\aliascntresetthe{lemma}

\newaliascnt{proposition}{theorem}
\newtheorem{proposition}[proposition]{Proposition}
\aliascntresetthe{proposition}

\newaliascnt{corollary}{theorem}
\newtheorem{corollary}[corollary]{Corollary}
\aliascntresetthe{corollary}

\newaliascnt{conjecture}{theorem}

\aliascntresetthe{conjecture}

\newaliascnt{example}{theorem}
\newtheorem{example}[example]{Example}
\aliascntresetthe{example}

\makeatletter
\def\tagform@#1{\maketag@@@{\ignorespaces#1\unskip\@@italiccorr}}
\let\orgtheequation\theequation
\def\theequation{(\orgtheequation)}
\makeatother
\def\equationautorefname~{}
\newcommand{\arxiv}[1]{%
 \href{http://front.math.ucdavis.edu/#1}{ArXiv:#1}}

\DeclareMathOperator*{\argmax}{\arg\!\max}
\DeclareMathOperator*{\argmin}{\arg\!\min}

\newcommand{\N}{{\mathbb N}}
\newcommand{\cN}{\mathcal{N}}
\newcommand{\R}{{\mathbb R}}

\newcommand{\cS}{\mathcal{S}}
\newcommand{\Z}{{\mathbb Z}}
\newcommand{\area}{\operatorname{Area}}
\newcommand{\ud}{\,\mathrm{d}}
\begin{document}

\title[Lattice points under convex curves]{Optimal stretching for lattice points under convex curves}
\author[]{Sinan Ariturk and Richard S. Laugesen}
\address{Pontif\'icia Universidade Cat\'olica do Rio de Janeiro, Brazil}
\email{ariturk\@@mat.puc-rio.br}
\address{Department of Mathematics, University of Illinois, Urbana,
IL 61801, U.S.A.}
\email{Laugesen\@@illinois.edu}
\date{\today}


\begin{abstract}
Suppose we count the positive integer lattice points beneath a convex decreasing curve in the first quadrant having equal intercepts. Then stretch in the coordinate directions so as to preserve the area under the curve, and again count  lattice points. Which choice of stretch factor will maximize the lattice point count? We show the optimal stretch factor approaches $1$ as the area approaches infinity. In particular, when $0<p<1$, among $p$-ellipses $|sx|^p+|s^{-1}y|^p=r^p$ with $s>0$, the one enclosing the most first-quadrant lattice points approaches a $p$-circle ($s=1$) as $r \to \infty$. 

The case $p=2$ was established by Antunes and Freitas, with generalization to $1<p<\infty$ by Laugesen and Liu. The case $p=1$ remains open, where the question is: which right triangles in the first quadrant with two sides along the axes will enclose the most lattice points, as the area tends to infinity? 

Our results for $p<1$ lend support to the conjecture that in all dimensions, the rectangular box of given volume that minimizes the $n$-th eigenvalue of the Dirichlet Laplacian will approach a cube as $n \to \infty$. This conjecture remains open in dimensions four and higher. 
\end{abstract}

\maketitle

\section{\bf Introduction}

This article tackles a variant of the Gauss circle problem motivated by shape optimization results for eigenvalues of the Laplacian, as explained in the next section. The circle problem asks for good estimates on the number of integer lattice points contained in a circle of radius $r>0$. Gauss showed this lattice point count equals the area of the circle plus an error of magnitude $O(r)$ as $r \to \infty$. The current best estimate, due to Huxley \cite{Hux03}, improves the error bound to $O(r^{\theta+\epsilon})$ for $\theta = 131/208$, which is still quite far from the exponent  $\theta =1/2$ conjectured by Hardy \cite{H15}. 

One may count lattice points inside other curves than circles, and may further seek to maximize the number of lattice points with respect to families of curves all enclosing the same area. Such maximization problems are the focus of this paper, concerning curves and lattice points in the first quadrant. 

Consider a convex decreasing curve in the first quadrant that intercepts the horizontal and vertical axes. For example, fix $0<p<1$ and consider the $p$-ellipse 
\begin{equation}
\label{superellipse}
	(sx)^p + (y/s)^p = r^p ,
\end{equation}
where $r,s>0$. This $p$-ellipse is obtained by stretching the $p$-circle from \autoref{fig:p-ellipsefig} in the coordinate directions by factors $s$ and $s^{-1}$ and then dilating by the scale factor $r$. 
\begin{figure}
\includegraphics[scale=0.35]{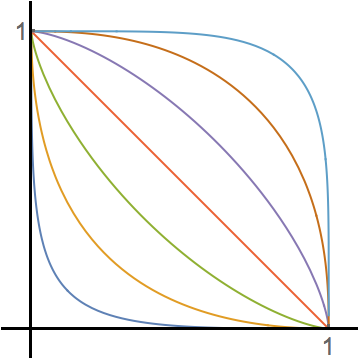}
\caption{\label{fig:p-ellipsefig}The family of $p$-circles $x^p+y^p=1$. \autoref{th:main1} and \autoref{ex:p-ellipse} handle the convex case $0<p<1$. The concave case $1<p<\infty$ was treated in \cite{LL16}. The straight line case $p=1$ remains open.}
\end{figure}
Note the $p$-ellipse has semi-axes $rs$ and $rs^{-1}$, and has area $A(r)$ depending only on the ``radius'' $r$, not on the stretch parameter $s$. Write $N(r,s)$ for the number of positive-integer lattice points lying below the curve, and for each fixed $r$, denote by $S(r)$ the set of $s$-values maximizing $N(r,s)$. In other words, $s \in S(r)$ maximizes the first-quadrant lattice point count among all $p$-ellipses having area $A(r)$. 

Our main theorem implies that these maximizing $s$-values converge to $1$ as $r$ goes to infinity. That is, the $p$-ellipses that contain the most positive-integer lattice points must have semi-axes of almost equal length, for large $r$, and thus can be described as ``asymptotically balanced''. This result in \autoref{ex:p-ellipse} is an application of \autoref{th:main1}, which handles much more general convex decreasing curves.

For nonnegative-integer lattice points, meaning we include also the lattice points on the axes, the problem is to minimize rather than maximize the number of enclosed lattice points. For that problem too we prove optimal curves are asymptotically balanced.

A key step in the proof is to establish a precise estimate on the number of positive-integer lattice points under the graph of a convex decreasing function, in \autoref{th:asymptotic}. This estimate builds on the corresponding estimate for concave functions, namely the work of Laugesen and Liu \cite{LL16} on the case $1<p<\infty$ and generalizations, which in turn was based on work of Kr\"{a}tzel \cite{kratzel04}. Our proof starts by observing that the convex and concave problems are complementary, as one sees by enclosing the convex curve in a suitable rectangle and regarding the lattice points above the curve as being lattice points beneath the ``upside down'' concave curve. 

\section{\bf Eigenvalues of the Laplacian, and open problems}

In this expository section we connect lattice point counting results to shape optimization problems on eigenvalues of the Laplacian. Open problems for eigenvalues arise naturally in this context. 

\subsection*{\bf Eigenvalues of the Laplacian} The asymptotic counting function maximization problem was initiated by Antunes and Freitas \cite{AF13}, who solved the problem for positive-integer lattice points inside standard ellipses. That is, they established the case  $p=2$ of the previous section. Their result was formulated in terms of shape optimization for Laplace eigenvalues, as we proceed to explain. 

For a bounded domain $\Omega \subset \R^d$, the eigenvalue problem for the Laplacian with Dirichlet boundary conditions is:
\[
	\begin{cases}
		-\Delta u = \lambda u & \text{in } \Omega , \\
		\hspace{0.7cm} u = 0 & \text{on } \partial \Omega ,
	\end{cases}
\]
where the eigenvalues form an increasing sequence
\[
	0 < \lambda_1(\Omega) \le \lambda_2(\Omega) \le \lambda_3(\Omega) \le \dots .
\]
The relationship between the domain $\Omega$ and its associated eigenvalues is complicated.
A classical problem is to determine the domain having given volume that minimizes the $n$-th eigenvalue.
A ball minimizes the first eigenvalue, by the Faber--Krahn inequality, and the union of two disjoint balls having the same radius minimizes the second eigenvalue, by the Krahn--Szego inequality. Domains that minimize higher eigenvalues do exist \cite{B12,BH00}, although the minimizing domains are not known explicitly. In two dimensions, a disk is conjectured to minimize the third eigenvalue, and more generally it is an open problem to determine whether a ball in $d$ dimensions minimizes the $(d+1)$-st eigenvalue \cite[p. 82]{H06}. Minimizing domains have been studied numerically by Oudet \cite{Oud04}, Antunes and Freitas \cite{AF12}, and Antunes and Oudet \cite{AO}, \cite[Chapter~11]{H17}. 

A challenging open problem is to determine the asymptotic behavior as $n \to \infty$ of the domain (or domains) minimizing the $n$-th eigenvalue. To gain insight, let us write $M(\lambda)$ for the number of eigenvalues less than or equal to the parameter $\lambda$, and recall that the Weyl conjecture claims
\[
	M(\lambda) = \omega_d (2\pi)^{-d} | \Omega | \lambda^{d/2} - (1/4) \omega_{d-1} (2\pi)^{1-d} | \partial \Omega | \lambda^{(d-1)/2} + o(\lambda^{(d-1)/2}) 
\]
where $\omega_d$ is the volume of the unit ball in $\R^d$. This asymptotic formula for the counting function was verified by Ivrii \cite{I80} under a generic assumption for piecewise smooth domains, namely that the periodic billiards have measure zero. The appearance of the perimeter in the second term of this formula might suggest that the domain minimizing the $n$-th eigenvalue (or maximizing the counting function $M(\lambda)$), under our assumption of fixed volume, should converge to a ball because the ball has minimal perimeter by the isoperimetric theorem. 

This heuristic does not amount to a proof, though, since the order of operations is wrong: our task is not to fix a domain and then let $n \to \infty$ ($\lambda \to \infty$), but rather to minimize the eigenvalue over all domains for $n$ fixed (maximize the counting function for $\lambda$ fixed) and only then let $n \to \infty$ ($\lambda \to \infty$).

It is an open problem to determine whether the eigenvalue-minimizing domain converges to a ball as $n \to \infty$. The problem is easier if the perimeter is fixed, and in that case Bucur and Freitas \cite{BF13} showed that eigenvalue minimizing domains do indeed converge to a disk, in dimension two.

Antunes and Freitas \cite{AF13} solved the problem in the class of rectangles under area normalization, as follows. Let $R(s)$ be the rectangle $(0,\pi/s) \times (0, s \pi)$, whose area equals $\pi^2$ for all $s$. For each $n$, choose a number $s_n>0$ such that $R(s_n)$ minimizes the $n$-th Dirichlet eigenvalue of the Laplacian. That is, choose $s_n$ such that
\[
	\lambda_n \big( R(s_n) \big) = \min_{s > 0} \lambda_n \big( R(s) \big) .
\]
Antunes and Freitas showed $s_n \to 1$ as $n \to \infty$, meaning that the rectangles $R(s_n)$ converge to a square. The analogous result for three-dimensional rectangular boxes was later established by van den Berg and Gittins \cite{BG16}. The problem remains open in dimensions four and higher. Once again, the problem is easier if the surface area is fixed, and in that case Antunes and Freitas \cite{AF16} showed that rectangular boxes which minimize the $n$-th Dirichlet eigenvalue of the Laplacian must converge to a cube, in any dimension.

The eigenvalues of the Laplacian on a rectangle are closely connected to lattice point counting: the eigenfunction $u=\sin(jsx) \sin(ky/s)$ on the rectangle $R(s)$ has eigenvalue $\lambda = (js)^2+(k/s)^2$, for $j,k>0$, and this eigenvalue is less than or equal to some number $r^2$ if and only if the lattice point $(j,k)$ lies inside the ellipse with semi-axes $s^{-1}$ and $s$ and radius $r$. Thus the result of Antunes and Freitas on asymptotically minimizing the $n$-th eigenvalue among rectangles of given area is essentially equivalent to asymptotically maximizing the number of first-quadrant lattice points enclosed by ellipses of given area --- and that is how their proof proceeded. 

\subsection*{A conjecture on product domains}
The conjecture for rectangular boxes in higher dimensions is supported by results in this paper, as follows. More generally, fix a bounded domain $\Omega \subset \R^d$, and for $s>0$ define a product domain 
\[
P(s) = (s^{-1/d} \Omega) \times (s^{1/d} \Omega) \subset \R^{2d} .
\]
For each $n$, choose $s_n$ to minimize the $n$-th Dirichlet eigenvalue of the Laplacian on the product domain.
It is natural to ask whether $s_n \to 1$ as $n \to \infty$, and our results suggest this might be the case.

Observe that the eigenvalues of $P(s)$ are given by $s^{2/d} \lambda_j(\Omega) + s^{-2/d} \lambda_k(\Omega)$ for $j,k > 0$. Without loss of generality, assume $\Omega$ has volume $(2\pi)^d/\omega_d$. 
Then the first-order Weyl approximation is $\lambda_n(\Omega) \sim n^{2/d}$. Using this approximation, we may approximate the eigenvalues of $P(s)$ by $s^{2/d} j^{2/d} + s^{-2/d} k^{2/d}$. That is, for $r>0$ the number of ``approximate eigenvalues'' less than $r^{2/d}$ is given by the number of positive-integer lattice points inside the $p$-ellipse \autoref{superellipse}, with $p=2/d$. 

If $d \ge 3$ then $p=2/d<1$, and so our \autoref{ex:p-ellipse} applies to the approximate eigenvalues.
Thus if $s_n$ were chosen to minimize the $n$-th ``approximate eigenvalue'' of $P(s)$, then $s_n$ would converge to $1$ as $n \to \infty$. This observation suggests the same might hold true for the $s_n$-value minimizing the actual $n$-th eigenvalue of the product domain. In particular, it seems reasonable to believe that the analogue of the Antunes--Freitas (and van den Berg--Gittins) result will hold for rectangular boxes in all even dimensions $\geq 6$, and presumably also in odd dimensions $\geq 5$. The evidence is hardly conclusive, of course, since not every rectangular box has the product form $P(s)$ and furthermore we used only the leading order term in the Weyl asymptotic. 

The preceding argument does not apply in $4$ dimensions: even if a $4$-dimensional box can be expressed as a product of two $2$-dimensional boxes, taking $d=2$ gives the borderline case $p=2/d=1$, and for $p=1$ the lattice point maximizing value $s_n$ does not seem to approach $1$ as $n \to \infty$ \cite[Section~9]{LL16}. Thus one might expect the conjecture on rectangular boxes to be hardest to prove in $4$ dimensions. 

\subsection*{More general domains}
Among more general convex domains with just a little regularity, Larson \cite{L16} has shown the ball asymptotically maximizes the Riesz means of the Laplace eigenvalues, for Riesz exponents $\geq 3/2$ in all dimensions. If the exponent could be lowered to $0$ in this result, then the ball would asymptotically maximize the counting function of individual eigenvalues. Incidentally, Larson also shows the cube is asymptotically optimal among polytopes, for the Riesz means. 

Thus the current state of knowledge is that asymptotic optimality holds for the individual eigenvalues if one restricts to rectangular boxes in $2$ or $3$ dimensions, and holds among more general convex domains and polytopes if one restricts to weaker eigenvalue functionals, namely the Riesz means of exponent $\geq 3/2$.

\section{\bf Assumptions and definitions}
\label{sec:assumptions}

By convention, the first quadrant is the \emph{open} set $\{ (x,y) : x,y>0 \}$. Take $\Gamma$ throughout the paper to be a convex, strictly decreasing curve in the first quadrant that intercepts the $x$- and $y$-axes at $x=L$ and $y=M$, as illustrated in \autoref{fig:gammafig}. Write $\area(\Gamma)$ for the area enclosed by the curve $\Gamma$ and the $x$- and $y$-axes. 

\begin{figure}
\includegraphics[scale=0.35]{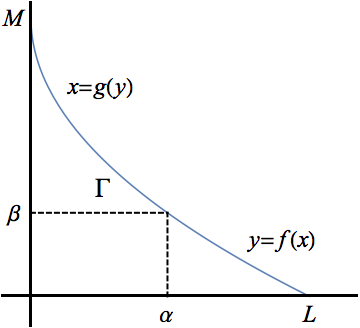}
\caption{\label{fig:gammafig}A convex decreasing curve $\Gamma$ in the first quadrant, with intercepts $L$ and $M$. The point $(\alpha,\beta)$ arises in \autoref{th:main1}.}
\end{figure}

Represent the curve as the graph of $y=f(x)$, so that $f$ is a convex strictly decreasing function for $x\in [0,L]$, and  
\[
M=f(0)>f(x)>f(L)=0 \qquad \text{whenever $x\in (0,L)$.}
\]
Denote the inverse function of $f(x)$ by $g(y)$ for $y \in[0,M]$. Clearly $g$ is also convex and strictly decreasing.

Compress the curve by a factor of $s>0$ in the horizontal direction and stretch it by the same factor in the vertical direction to obtain the curve 
\[
\Gamma(s) = \text{graph of $sf({sx})$.}
\]
The area under $\Gamma(s)$ equals the area under $\Gamma$. Then scale the curve by parameter $r>0$ to obtain:
\begin{align*}
r\Gamma(s)
& = \text{image of $\Gamma(s)$ under the radial scaling $(x,y) \mapsto (rx,ry)$} \\
& = \text{graph of $rsf({sx/r})$.}
\end{align*}
Define the counting function for 
$r\Gamma(s)$ by 
\begin{align*}
N(r,s) &=\text{number of positive-integer lattice points lying inside or on $r\Gamma(s)$ } \\
&=\# \big\{ (j,k)\in\mathbb{N} \times \mathbb{N}:k\leq rsf(js/r) \big\}. 
\end{align*}
For each $r>0$, we consider the set
\[
S(r) = \argmax_{s >0} N(r,s) 
\]
consisting of the $s$-values that maximize the number of first-quadrant lattice points enclosed by the curve $r\Gamma(s)$. The set $S(r)$ is well-defined because for each fixed $r$, the counting function $N(r,s)$ equals zero whenever $s$ is sufficiently large or sufficiently close to $0$.

\section{\bf Results}
\label{sec:mainresults}

Recall $g$ is the inverse function of $f$, as illustrated in \autoref{fig:gammafig}. 
\begin{theorem}[Optimal convex curve is asymptotically balanced]\label{th:main1}
Assume $(\alpha,\beta) \in \Gamma$ is a point in the first quadrant with $\alpha < L/2, \beta < M/2$, such that $f \in C^2[\alpha,L)$ with $f'<0$ and $f''>0$ on $[\alpha,L)$, and similarly $g\in C^2[\beta,M)$ with $g'<0$ and $g''>0$ on $[\beta,M)$. Further suppose there is a partition $\alpha=\alpha_0 < \alpha_1 < \ldots < \alpha_m = L$ such that $f''$ is monotonic over each subinterval $(\alpha_i,\alpha_{i+1})$, and a partition $\beta=\beta_0 < \beta_1 < \ldots < \beta_n = M$ such that $g''$ is monotonic over each subinterval $(\beta_i,\beta_{i+1})$. Moreover, assume there are functions $\delta:(0, \infty) \to (0, L/2 - \alpha)$ and $\epsilon:(0, \infty) \to (0, M/2-\beta)$ and positive constants $a_1, a_2, b_1, b_2$ such that as $r \to \infty$,
\begin{align*}
	\delta(r) = O(r^{-2a_1}) , \quad \quad \frac{1}{f'' \big( L-\delta(r) \big) } = O(r^{1-4a_2}) , \\
	\epsilon(r) = O(r^{-2b_1}) , \quad \quad \frac{1}{g'' \big( M-\epsilon(r) \big)} = O(r^{1-4b_2}) .
\end{align*}

Assume the intercepts of $\Gamma$ are equal ($L=M$). Then the optimal stretch factor for maximizing $N(r,s)$ approaches $1$ as $r$ tends to infinity, with
\[
S(r) \subset \big[ 1-O(r^{-e}),1+O(r^{-e}) \big] 
\]
where the exponent is $e=\min(\frac{1}{6}, a_1, a_2, b_1, b_2)$. Further, the maximal lattice count has asymptotic formula 
\[
\max_{s > 0} N(r,s) = r^2\area (\Gamma)-rL + O(r^{1-2e}) .
\]
\end{theorem}
The theorem is proved in \autoref{sec:mainproof}. The $C^2$-smoothness hypothesis can be weakened to piecewise smoothness, cf.\ \cite{LL16}, although for simplicity we will not do so here. 

The theorem simplifies considerably when the second derivatives are positive and monotonic all the way up to the endpoints:
\begin{corollary}\label{co:main1}
Assume $(\alpha,\beta) \in \Gamma$ is a point in the first quadrant with $\alpha < L/2, \beta < M/2$, such that $f \in C^2[\alpha,L]$ with $f'<0,f''>0$ and $f''$ monotonic, and $g\in C^2[\beta,M]$ with $g'<0, g''>0$ and $g''$ monotonic.

If the intercepts of $\Gamma$ are equal ($L=M$), then the optimal stretch factor for maximizing $N(r,s)$ approaches $1$ as $r$ tends to infinity, with
\[
S(r) \subset \big[ 1-O(r^{-1/6}),1+O(r^{-1/6}) \big] ,
\]
and the maximal lattice count satisfies $\max_{s > 0} N(r,s) = r^2\area (\Gamma)-rL + O(r^{2/3})$. 
\end{corollary}
The corollary follows by taking $a_1=b_1=1/2, a_2=b_2=1/4,e=1/6$ in the theorem and noting that $f''(L)>0$ and $g''(M)>0$ by assumption. 

\begin{example}[Optimal $p$-ellipses for lattice point counting]\rm \label{ex:p-ellipse}
Fix $0<p<1$, and consider the $p$-circle
\[
\Gamma : |x|^p+|y|^p=1 ,
\]
which has equal intercepts $L=M=1$. That is, the $p$-circle is the unit circle for the $\ell^p$-metric on the plane. Then the $p$-ellipse
\[
r\Gamma(s) : |sx|^p + |s^{-1}y|^p \leq r^p
\]
has first-quadrant counting function 
\[
N(r,s) = \# \{ (j,k) \in \N \times \N :  (js)^p+(ks^{-1})^p \leq r^p \} .
\]

We will show that the $p$-ellipse containing the maximum number of positive-integer lattice points must approach a $p$-circle in the limit as $r \to \infty$, with 
\[
S(r) \subset [1-O(r^{-e}),1+O(r^{-e})] 
\]
where $e=\min \{ \tfrac{1}{6},\tfrac{p}{2} \}$.

To verify that the $p$-circle satisfies the hypotheses of \autoref{th:main1}, we let $\alpha=\beta=2^{-1/p}$, so that $\alpha < 1/2=L/2$ and $\beta<1/2=M/2$. Then for $0<x<1$ we have
\begin{align*}
f(x) & = (1-x^p)^{1/p} , \\
f'(x) & = -x^{p-1} (1-x^p)^{-1+1/p} < 0, \\
f''(x) & =(1-p)x^{p-2}(1-x^p)^{-2+1/p} > 0, \\
f'''(x) & = (1-p) x^{p-3} (1 - x^p)^{-3 + 1/p} \big( (1+p)x^p + p-2 \big) .
\end{align*}
If $0<p \leq 1/2$ then $f'''<0$ on the interval $(0,1)$, and so $f''$ is monotonic. If $1/2<p<1$ then $f'''$ vanishes at exactly one point in the interval $(\alpha,1)$, namely at $\alpha_1=[(2-p)/(1+p)]^{1/p}$, and so $f''$ is monotonic on the subintervals $(\alpha,\alpha_1)$ and $(\alpha_1,1)$. Further, we choose $a_1=a_2=p/2$ and let $\delta(r)=r^{-2a_1}=r^{-p}$ for all large $r$, and verify directly that 
\[
\frac{1}{f'' \big( 1-\delta(r) \big) } = O(r^{1-2p}) = O(r^{1-4a_2}) .
\]
The calculations are the same for $g$, and so the desired conclusion for $p$-ellipses with $0<p<1$ now follows from \autoref{th:main1}. 

The case $1<p<\infty$ was treated earlier by Laugesen and Liu \cite[Example~4]{LL16}. They raised the case $p=1$ as an interesting open problem \cite[Section~9]{LL16}. Stated informally, one asks: which right triangle with two sides along the axes contains the most lattice points as $r \to \infty$?

\end{example}

\subsection*{Lattice points in the closed first quadrant} 
We consider also a similar problem concerning the number $\cN(r,s)$ of lattice points in the \emph{closed} first quadrant enclosed by the curve $r\Gamma(s)$. For $r>0$, define $\cS(r)$ to be the set of minimum points of the function $s \mapsto \cN(r,s)$. (Note the maximization problem has no solution, since one can enclose arbitrarily many points on the vertical axis by letting $s \to \infty$, or on the horizontal axis by letting $s \to 0$.) Under the same assumptions as \autoref{th:main1}, we show this minimizing set $\cS(r)$ converges to $\{ 1 \}$ as $r$ goes to infinity. 

To state this result precisely, let $\Z_+ = \{ 0, 1, 2, 3, \ldots \}$ and define a counting function  
\begin{align*}
\cN(r,s) &=\text{number of nonnegative-integer lattice points lying inside or on $r\Gamma(s)$ } \\
&=\# \big\{ (j,k)\in \Z_+ \times \Z_+:k\leq rsf(js/r) \big\}. 
\end{align*}
For each $r>0$, define the set
\[
\cS(r) = \argmin_{s >0} \cN(r,s) 
\]
consisting of the $s$-values that minimize the number of closed first-quadrant lattice points enclosed by the curve $r\Gamma(s)$.

\begin{theorem}[Optimal convex curve is asymptotically balanced]\label{th:main3}
Assume the hypotheses of \autoref{th:main1} hold and the intercepts of $\Gamma$ are equal ($L=M$). Then the optimal stretch factor for minimizing $\cN(r,s)$ approaches $1$ as $r$ tends to infinity, with
\[
\cS(r) \subset \big[ 1-O(r^{-e}),1+O(r^{-e}) \big] .
\]
Further, the minimal lattice count has asymptotic formula 
\[
\min_{s > 0} \cN(r,s) = r^2\area (\Gamma)+rL + O(r^{1-2e}) .
\]
\end{theorem}
The theorem holds in particular when the second derivatives of $f$ and $g$ are positive and monotonic all the way up to the endpoints, thus yielding a corollary analogous to \autoref{co:main1}. Also, \autoref{th:main3} applies in particular when the curve $\Gamma$ is a $p$-ellipse with $0<p<1$, since we verified the hypotheses already in \autoref{ex:p-ellipse}. 

Concave curves, such as $p$-ellipses with $1<p<\infty$, were handled earlier by Laugesen and Liu \cite{LL16}. The standard ellipse case ($p=2$) was done first by van den Berg, Bucur, and Gittins \cite{BBG16}, who used it to show that the rectangle of given area maximizing the $n$-th Neumann eigenvalue of the Laplacian will converge to a square as $n \to \infty$.

\section{\bf Two-term upper bound on counting function}

In order to control the stretch factor when proving our main results later in the paper, we now develop a two-term upper bound on the lattice point counting function. The leading order term of the bound is simply the area inside the curve, and thus is best possible, while the second term scales like the length of the curve and so has the correct order of magnitude. 

Recall $\Gamma$ is the graph of $y=f(x)$, where $f$ is convex and strictly decreasing on $[0,L]$, with $f(0)=M, f(L)=0$. We do not assume $f$ is differentiable, in the next proposition. 
\begin{proposition}[Two-term upper bound on counting function]\label{prop:counting_positive} The number $N(r,s)$ of positive-integer lattice points lying inside $r\Gamma(s)$ in the first quadrant satisfies
\[
N(r,s)\leq r^2\area(\Gamma) - \frac{1}{2} f \big( \frac{L}{2} \big) rs
\]
whenever $r \geq 2s/L$.
\end{proposition}
\begin{proof}
It is enough to prove the case $r=s=1$ for $L \geq 2$, because then the general case of the proposition follows by applying the special case to the curve $r\Gamma(s)$ (which has horizontal intercept $rs^{-1}L$ and defining curve $y=rsf(sx/r)$).

Clearly $N(1,1)$ equals the total area of the squares of sidelength $1$ having upper right vertices at positive integer lattice points inside the curve $\Gamma$. The union of these squares is contained in $\Gamma$, since the curve is decreasing.

\begin{figure}
\includegraphics[scale=0.4]{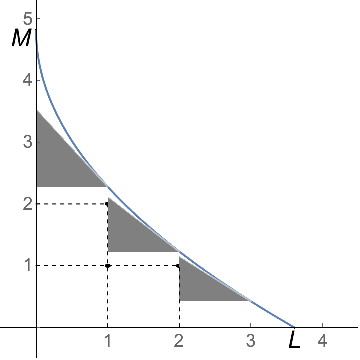}
\caption{\label{fig:counting_positive}Positive integer lattice count $N(1,1) \leq \area(\Gamma) - \area(\text{triangles})$, in proof of \autoref{prop:counting_positive}.}
\end{figure}

Consider the right triangles of width $1$ formed by left-tangent lines on $\Gamma$, as shown in \autoref{fig:counting_positive}. The triangles have vertices $\big( i-1,f(i) \big),\big( i,f(i) \big),\big( i-1,f(i)-f'(i-) \big)$, for $i=1,\ldots,\lfloor L \rfloor$. Clearly the triangles lie under the curve by concavity, and lie outside the union of squares. 

Hence
\[
N(1,1) \leq \area(\Gamma) - \area(\text{triangles}) .
\]
To complete the proof, we estimate as follows:
\begin{align*}
\area(\text{triangles}) 
& = \frac{1}{2} \sum_{i=1}^{\lfloor L \rfloor} |f'(i-)| \\
& \geq \bigg( \frac{1}{2} \sum_{i=1}^{\lfloor L \rfloor-1} \big( f(i)-f(i+1) \big) \bigg) + \frac{1}{2} \big( f(\lfloor L \rfloor) - f(L) \big) \qquad \text{by convexity} \\
& = \frac{1}{2} \big( f(1) - f(L) \big) \\
& \geq \frac{1}{2} f(L/2) 
\end{align*}
since $L/2 \geq 1$ and $f(L)=0$. 
\end{proof}

\section{\bf Two-term counting asymptotics with explicit remainder}
\label{sec:twoterm}

What matters in the following proposition is that the terms on the right side of the estimate in part (b) can be shown later to have order less than $O(r)$, and thus can be treated as remainder terms. Also, it matters that the $s$-dependence in the estimate can be seen explicitly.  

\begin{proposition}[Two-term counting estimate]\label{th:asymptotic}
%PC2
Assume $(\alpha,\beta) \in \Gamma$ is a point in the first quadrant such that $f \in C^2[\alpha,L)$ with $f'<0$ and $f''>0$ on $[\alpha,L)$, and similarly $g\in C^2[\beta,M)$ with $g'<0$ and $g''>0$ on $[\beta,M)$. Further suppose there is a partition $\alpha=\alpha_0 < \alpha_1 < \ldots < \alpha_m = L$ such that $f''$ is monotonic over each subinterval $(\alpha_i,\alpha_{i+1})$, and a partition $\beta=\beta_0 < \beta_1 < \ldots < \beta_n = M$ such that $g''$ is monotonic over each subinterval $(\beta_i,\beta_{i+1})$.

\noindent (a) Assume the curve $\Gamma$ does not pass through any integer lattice points. Suppose $\alpha < \lfloor L \rfloor$ and $\beta < \lfloor M \rfloor$, and let $0<\delta<\lfloor L \rfloor - \alpha$ and $0 < \epsilon < \lfloor M \rfloor - \beta$.
Then the number $N$ of positive-integer lattice points inside $\Gamma$ in the first quadrant satisfies:
\begin{align*}
&\big|N-\area (\Gamma)+(L+M)/2\big|\notag\\
&\leq 6\Big(\int_\alpha^L f''(x)^{1/3} \ud x+\int_\beta^M g''(y)^{1/3}\ud y\Big)  +175\big(\frac{1}{f''(L-\delta)^{1/2}}+\frac{1}{g''(M-\epsilon)^{1/2}}\big)\notag\\
&+525\big(\sum_{i=0}^{m-1} \frac{1}{f''(\alpha_i)^{1/2}}+\sum_{i=0}^{n-1} \frac{1}{g''(\beta_i)^{1/2}}\big) +\frac{1}{4} \big( \sum_{i=0}^{m-1} |f'(\alpha_i)| + \sum_{i=0}^{n-1} |g'(\beta_i)| \big) \notag\\
& \hspace{4cm}+ \frac{1}{2}(\delta+\epsilon) + 3(m+n) + 5 + \frac{L}{M}+\frac{M}{L} . \notag\\
\end{align*}

\noindent (b) Suppose $\alpha < L/2$ and $\beta < M/2$, and let $\delta:(0,\infty) \to (0, L/2-\alpha)$ and $\epsilon:(0,\infty) \to (0, M/2-\beta)$ be functions.
The number $N(r,s)$ of positive-integer lattice points lying inside $r\Gamma(s)$ in the first quadrant satisfies (for all $r,s>0$ such that $rs^{-1}L \ge 1$ and $rsM \ge 1$):
\begin{align*}
&\big|N(r,s)-r^2\area(\Gamma)+r(s^{-1}L+sM)/2\big|\notag\\
& \leq 6r^{2/3}\Big(\int_\alpha^L f''(x)^{1/3} \ud x+\int_\beta^M g''(y)^{1/3}\ud y\Big)\\
& \quad +175r^{1/2}\big(\frac{s^{-3/2}}{f''(L-\delta(r))^{1/2}}+\frac{s^{3/2}}{g''(M-\epsilon(r))^{1/2}}\big)\notag\\
&\quad +525r^{1/2}\big(\sum_{i=0}^{m-1} \frac{s^{-3/2}}{f''(\alpha_i)^{1/2}}+\sum_{i=0}^{n-1} \frac{s^{3/2}}{g''(\beta_i)^{1/2}}\big)
+\frac{1}{4} \big( \sum_{i=0}^{m-1} s^2 |f'(\alpha_i)| + \sum_{i=0}^{n-1} s^{-2} |g'(\beta_i)| \big)\notag\\
&\quad + \frac{1}{2}r \big( s^{-1}\delta(r)+s\epsilon(r) \big)+3(m+n)+5 + s^{-2} \frac{L}{M} + s^2 \frac{M}{L} . \notag
\end{align*}
\end{proposition}
Notice the integral of $(f'')^{1/3}$ in the Proposition is finite, because it is bounded by a constant times 
\[
\big( \int_\alpha^L f''(x) \ud x \big)^{1/3} = \big( f'(L-) - f'(\alpha) \big)^{1/3} \leq \big( - f'(\alpha) \big)^{1/3} < \infty.
\]
The integral of $(g'')^{1/3}$ is similarly finite. 
\begin{proof}
Part (a). In what follows, remember $L$ and $M$ are not integers since $\Gamma$ is assumed not to pass through any integer lattice points. 

The idea is to count lattice points in the ``complementary region'' lying above the convex curve $\Gamma$ and inside the rectangle $\big[ 0,\lfloor L \rfloor \big] \times \big[ 0, \lfloor M \rfloor \big]$, because then one may invoke known estimates for a region with concave boundary, \emph{e.g.} \cite[Proposition~8]{LL16}. The complementary region is shown in \autoref{fig:gammafigcomp}. 
\begin{figure}
\includegraphics[scale=0.40]{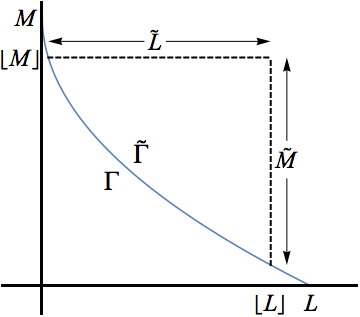}
\caption{\label{fig:gammafigcomp}The convex decreasing curve $\Gamma$, and its complementary curve $\widetilde{\Gamma}$, which is concave decreasing with respect to an origin at the point $\big( \lfloor L \rfloor , \lfloor M \rfloor \big)$.}
\end{figure}
Its width and height are
\[
	\widetilde{L} = \lfloor L \rfloor - g(\lfloor M \rfloor), \qquad \widetilde{M} = \lfloor M \rfloor - f(\lfloor L \rfloor) ,
\]
and we define strictly decreasing functions $F:[0, \widetilde{L}] \to [0, \widetilde{M}]$ and $G:[0, \widetilde{M}] \to [0,\widetilde{L}]$ by
\[
F(x) = \lfloor M \rfloor - f \big( \lfloor L \rfloor -x \big) , \qquad G(y) = \lfloor L \rfloor - g \big( \lfloor M \rfloor -y \big) .
\]
Notice $F$ and $G$ are inverses, with $y=F(x)$ if and only if $x=G(y)$. 

Write $\widetilde{\Gamma}$ for the graph of $F$ (or $G$), so that $\widetilde{\Gamma}$ decreases from its $y$-intercept at $(0,\widetilde{M})$ to its $x$-intercept at $(\widetilde{L},0)$. 
Define $\widetilde{\alpha} = \lfloor L \rfloor - \alpha$ and $\widetilde{\beta} = \lfloor M \rfloor - \beta$.
Then $\widetilde{\alpha} > 0$ because we assumed $\alpha < \lfloor L \rfloor$. Applying $f$ to both sides of this inequality gives $\beta > f(\lfloor L \rfloor) = \lfloor M \rfloor - \widetilde{M}$, and so $\widetilde{\beta} < \widetilde{M}$. Similarly, we find $\widetilde{\beta} > 0$ and $\widetilde{\alpha} < \widetilde{L}$. Also, $0<\delta<\widetilde{\alpha}$ and $0<\epsilon<\widetilde{\beta}$ by the hypotheses in Part (a).

Note $(\widetilde{\alpha} ,\widetilde{\beta}) \in \widetilde{\Gamma}$ with $F(\widetilde{\alpha})=\widetilde{\beta}$.  Clearly $F \in C^2[0,\widetilde{\alpha}]$ with $F'<0$ and $F''<0$ on $[0,\widetilde{\alpha}]$, and similarly $G \in C^2[0,\widetilde{\beta}]$ with $G'<0$ and $G''<0$ on $[0,\widetilde{\beta}]$. Further, there is a partition $0 = \widetilde{\alpha}_0 < \widetilde{\alpha}_1 < \ldots < \widetilde{\alpha}_l = \widetilde{\alpha}$ such that $F''$ is monotonic on each subinterval $(\widetilde{\alpha}_i, \widetilde{\alpha}_{i+1})$.
This partition may be chosen so that $l \le m$ and $\widetilde{\alpha}_i = \lfloor L \rfloor - \alpha_{l-i}$ for $i=1,2,\ldots, l$. Likewise, there is a partition $0 = \widetilde{\beta}_0 < \widetilde{\beta}_1 < \ldots < \widetilde{\beta}_\ell = \widetilde{\beta}$ such that $G''$ is monotonic on each subinterval $(\widetilde{\beta}_i, \widetilde{\beta}_{i+1})$. This partition may be chosen so that $\ell \le n$ and $\widetilde{\beta}_i = \lfloor M \rfloor - \beta_{\ell-i}$ for $i=1,2,\ldots, \ell$.

Let $\widetilde{N}$ be the number of positive-integer lattice points bounded by $\widetilde{\Gamma}$.
Then by \cite[Proposition~8(a)]{LL16} applied to the concave decreasing curve $\widetilde{\Gamma}$, we have
\begin{align*}
&\big|\widetilde{N}-\area (\widetilde{\Gamma})+(\widetilde{L}+\widetilde{M})/2\big|\notag\\
&\leq 6\Big(\int_0^{\widetilde{\alpha}} |F''(x)|^{1/3} \ud x+\int_0^{\widetilde{\beta}} |G''(y)|^{1/3}\ud y\Big)  +175\big(\frac{1}{|F''(\delta)|^{1/2}}+\frac{1}{|G''(\epsilon)|^{1/2}}\big)\notag\\
&+350\big(\sum_{i=1}^l \frac{1}{|F''(\widetilde{\alpha_i})|^{1/2}}+\sum_{i=1}^\ell \frac{1}{|G''(\widetilde{\beta_i})|^{1/2}}\big) +\frac{1}{4} \big( \sum_{i=1}^l |F'(\widetilde{\alpha_i})| + \sum_{i=1}^\ell |G'(\widetilde{\beta_i})| \big) \notag\\
& \hspace{4cm}+(\delta+\epsilon)/2 + 3(l + \ell) + 1 . \notag\\
\end{align*}

Counting positive-integer lattice points in the rectangle $\big[ 0,\lfloor L \rfloor \big] \times \big[ 0, \lfloor M \rfloor \big]$ gives
\[
\lfloor L \rfloor \, \lfloor M \rfloor = N + \widetilde{N} + \lfloor \widetilde{L} \rfloor + \lfloor \widetilde{M} \rfloor - 1 .
\]
(Both $N$ and $\widetilde{N}$ include in their count any positive-integer lattice points lying on the curve $\Gamma$. Such double-counting is avoided, though, because the curve is assumed to contain no such lattice points.) The area of the rectangle can be decomposed as
\[
\lfloor L \rfloor \, \lfloor M \rfloor =  \area(\Gamma) + \area(\widetilde{\Gamma}) - \area(U_L) - \area(U_M) 
\]
where $U_L$ is the region bounded by the curve $\Gamma$, the $x$-axis, and the line $x=\lfloor L \rfloor$, and $U_M$ is the region bounded by $\Gamma$, the $y$-axis and the line $y=\lfloor M \rfloor$. After equating the last two displayed equations, we conclude
\begin{align*}
& \Big| N - \area(\Gamma)+(L+M)/2 + \widetilde{N} - \area(\widetilde{\Gamma})+(\widetilde{L}+\widetilde{M})/2 \Big| \\
& \leq \Big| \frac{L+M}{2} + \frac{\widetilde{L}+\widetilde{M}}{2} - \lfloor \widetilde{L} \rfloor - \lfloor \widetilde{M} \rfloor \Big| + 1 + \area(U_L) + \area(U_M) \\
& \leq \Big| \frac{\lfloor L \rfloor - \widetilde{L}}{2} + \frac{\lfloor M \rfloor - \widetilde{M}}{2} \Big| + 4 + \area(U_L) + \area(U_M) .
\end{align*}

By convexity, $U_L$ is contained in a right triangle of width $L-\lfloor L \rfloor \le 1$ and height $f(\lfloor L \rfloor) \le M/L$. Similarly, $U_M$ is contained in a right triangle of height $M-\lfloor M \rfloor \le 1$ and width $g( \lfloor M \rfloor) \le L/M$. Hence,
\[
	\area(U_L) + \area(U_M) \le \frac{1}{2} \bigg( \frac{L}{M} + \frac{M}{L} \bigg) .
\]
Also $\lfloor L \rfloor - \widetilde{L} = g( \lfloor M \rfloor) \le L/M$ and $\lfloor M \rfloor - \widetilde{M} = f( \lfloor L \rfloor) \le M/L$. Combining these results, we conclude 
\begin{align*}
\Big| N - \area(\Gamma)+(L+M)/2 \Big| \leq \Big| \widetilde{N} - \area(\widetilde{\Gamma})+(\widetilde{L}+\widetilde{M})/2 \Big| + 4 + \frac{L}{M} + \frac{M}{L} .
\end{align*}

To complete the proof from the above estimates, note that
\begin{align*}
	\int_0^{\widetilde{\alpha}} |F''(x)|^{1/3} \ud x+\int_0^{\widetilde{\beta}} |G''(y)|^{1/3}\ud y 
	& \le \int_\alpha^L f''(x)^{1/3} \ud x+\int_\beta^M g''(y)^{1/3}\ud y \\
	\sum_{i=1}^l \frac{1}{|F''(\widetilde{\alpha}_i)|^{1/2}}+ \sum_{i=1}^\ell \frac{1}{|G''(\widetilde{\beta}_i)|^{1/2}} 
	& \le \sum_{i=0}^{m-1} \frac{1}{f''(\alpha_i)^{1/2}}+\sum_{i=0}^{n-1} \frac{1}{g''(\beta_i)^{1/2}} \\
	\sum_{i=1}^l |F'(\widetilde{\alpha}_i)|+ \sum_{i=1}^\ell |G'(\widetilde{\beta}_i)| 
	& \le \sum_{i=0}^{m-1} |f'(\alpha_i)| + \sum_{i=0}^{n-1} |g'(\beta_i)| 
\end{align*}
and
\begin{align*}
	& \quad \ \frac{1}{|F''(\delta)|^{1/2}}+\frac{1}{|G''(\epsilon)|^{1/2}} \\
	& \le \frac{1}{f''(L-\delta)^{1/2}} + \frac{1}{g''(M-\epsilon)^{1/2}} 
	+\sum_{i=0}^{m-1} \frac{1}{f''(\alpha_i)^{1/2}}+\sum_{i=0}^{n-1} \frac{1}{g''(\beta_i)^{1/2}} ,
\end{align*}
where the final inequality relies on the monotonicity assumptions on $f$ and $g$. 

\medskip
Part (b).
Apply Part (a) to the curve $r\Gamma(s)$ by replacing $L, M, f(x), g(y), \alpha, \beta,\delta,\epsilon$ with $rs^{-1}L, rsM, rsf(sx/r), rs^{-1}g(s^{-1}y/r), rs^{-1}\alpha, rs\beta,rs^{-1}\delta(r),rs\epsilon(r)$ respectively; we check the needed hypotheses for Part (a) as follows. The hypothesis ``$\alpha < \lfloor L \rfloor$'' is satisfied because 
\[
rs^{-1}\alpha < rs^{-1}L/2 \leq \lfloor rs^{-1}L \rfloor ,
\]
where we used the assumption $\alpha \leq L/2$ and the fact that $t/2 < \lfloor t \rfloor$ when $t \geq 1$. Similarly, the hypothesis ``$\delta + \alpha < \lfloor L \rfloor$'' in Part (a) is satisfied because
\[
rs^{-1}\delta(r) + rs^{-1}\alpha < rs^{-1}L/2 \leq \lfloor rs^{-1}L \rfloor .
\]
Hence from Part (a) we obtain the conclusion of Part (b) provided the curve $r\Gamma(s)$ does not pass through any integer lattice points. 

If the curve does pass through some lattice points, then simply consider a decreasing sequence $r_i \searrow r$ for which each curve $r_i \Gamma(s)$ contains no lattice points, and also modify the functions $\delta(\cdot)$ and $\epsilon(\cdot)$ to be continuous at $r$. Then the desired result follows by passing to the limit in the case of the theorem already proved, noting that $N(r,s) \leq N(r_i,s)$. 
\end{proof}

\section{\bf Elementary bounds on the optimal stretch factors}
\label{sec:rdepend}

We develop some $r$-dependent bounds on the optimal stretch factors. Later, in the proof of \autoref{th:main1}, we will show the stretch factors in fact converge to $1$.  
\begin{lemma}[$r$-dependent bound on optimal stretch factors]\label{lemma:bound}
If 
\[
r^2 \geq \frac{1}{\displaystyle \max_\Gamma xy}
\]
then
\[
S(r)\subset \big[(rM)^{-1}, rL\big] .
\]
\end{lemma}
\begin{proof}
Fix $r$, then let $(x_0,y_0) \in \Gamma$ be a point maximizing the product $xy$, and choose $s_0 = rx_0$. Then the curve $r\Gamma(s_0)$ passes through the point
\[
	\big(1, rs_0 f(s_0/r) \big) = (1, r^2 x_0 y_0) .
\]
By assumption $r^2 \geq 1/x_0 y_0$, and so the curve $r\Gamma(s_0)$ encloses the  point $(1,1)$. Hence the maximum of the counting function $s \mapsto N(r,s)$ is greater than zero. We will use that fact to constrain the $s$-values where the maximum can be attained. 

The curve $r\Gamma(s)$ has $x$-intercept at $rs^{-1}L$, which is less than $1$ if $s>rL$ and so in that case the curve encloses no positive-integer lattice points. Similarly if $s<(rM)^{-1}$, then $r\Gamma(s)$ has height less than $1$ and contains no lattice points in the first quadrant. The integer-valued function $s \mapsto N(r,s)$ is clearly bounded, and we saw in the first part of the proof that it is positive for some choice of $s_0$. Thus $N(r,s)$ attains its positive maximum at some $s$-value between $(rM)^{-1}$ and $rL$.
\end{proof}

\begin{lemma}[Improved $r$-dependent bound on optimal stretch factors]\label{lemma:bound2}
A constant $C$ exists, depending only on the curve $\Gamma$, such that if $r \geq C$ then
\[
	S(r) \subset \big[2(rM)^{-1}, \frac{1}{2}rL \big] .
\]
\end{lemma}
\begin{proof}
%The case $k=1$ is established in the previous lemma.
%By induction, we may assume that the lemma holds for $k$ and prove it for $k+1$.
%Let $r > C_k$ and fix $s$ in $S(r)$.
%By the inductive assumption,
%\[
%	s \le \frac{rL}{k}
%\]
%Suppose that $s > rL/(k+1)$.
%In particular,
%\[
%	k < \frac{Lr}{s} \le k+1
%\]
%Hence,
%\[
%	N(r,s) = \sum_{i=1}^k \, \lfloor rs f(sj/r) \rfloor
%\]
%Also,
%\[
%	N(r,s/2) \ge \sum_{j=1}^{2k} \, \lfloor rs/2 f(sj/2r) \rfloor
%\]
%Define
%\[
%	\delta_1 = \min \bigg\{ \Big( \sum_{j=1}^{2k} (1/2) f(jx/2) \Big) - \Big( \sum_{j=1}^k f(jx) \Big) : L/(k+1) \le x \le L/k \bigg \}
%\]
%and
%\[
%	\delta_2 = \min \bigg\{ \Big( \sum_{j=1}^{2k} (1/2) g(jy/2) \Big) - \Big( \sum_{j=1}^k g(jy) \Big) : M/(k+1) \le y \le M/k \bigg \}
%\]
%Let $\delta = \min(\delta_1, \delta_2)$ and let
%\[
%	C_{k+1} = \max \Big( \sqrt{\frac{2k(k+1)}{\delta L}} , C_k \Big)
%\]
%Then
%\[
%	\sum_{j=1}^{2k} (1/2) f(js/(2r)) \ge \delta + \sum_{j=1}^k f(js/r)
%\]
%Hence
%\[
%	\sum_{j=1}^{2k} (rs/2) f(js/(2r)) \ge rs \delta + \sum_{j=1}^k rs f(js/r)
%\]
%We have $rs \delta > 2k$, so
%\[
%	\sum_{j=1}^{2k} \lfloor (rs/2) f(js/(2r)) \rfloor > \sum_{j=1}^k \lfloor rs f(js/r) \rfloor
%\]
%This shows that $N(r,s/2) > N(r,s)$, which is a contradiction, because $s$ is in $S(r)$.
%This proves that $s \le rL/(k+1)$, and a symmetric argument can be used to show that $s \ge (k+1)/(rM)$.
%
Let $C=\max \big( \sqrt{8/L \delta_1},\sqrt{8/M \delta_2} \, \big)$ where 
\begin{align*}
	\delta_1 & = \min \Big\{ f(x/2)-f(x) : L/2 \le x \le L \Big \} , \\
	\delta_2 & = \min \Big\{ g(y/2)-g(y) : M/2 \le x \le M \Big \} .
\end{align*}
Choosing $x=L$ implies 
\[
(L/2) \delta_1 \leq (L/2) f(L/2) \leq \max_\Gamma xy ,
\]
and so $C^2 \geq 4/\max_\Gamma xy$. 

Fix $r \geq C$. Then $S(r)\subset \big[(rM)^{-1}, rL\big]$ by \autoref{lemma:bound}. To show $S(r)$ is contained in a smaller interval, we will show $s \notin S(r)$ when $s \in ( \frac{1}{2} rL,rL]$. So suppose in what follows that
\[
\frac{L}{2} < \frac{s}{r} \leq L . 
\]
We will prove $N(r,s) < N(r,s/2)$, which implies $s$ is not a maximizer for the counting function and so $s \notin S(r)$. 

By counting lattice points $(j,k)$ with $j=1$ and $j=2$, we find
\begin{align*}
N(r,s/2) 
& \geq \lfloor (rs/2) f(s/2r) \rfloor + \lfloor (rs/2) f(2s/2r) \rfloor \\
& > (rs/2) f(s/2r) +  (rs/2) f(s/r) - 2 \\
& \geq (rs/2) \delta_1 + rsf(s/r) - 2 \qquad \text{by taking $x=s/r$ in the definition of $\delta_1$} \\
& > rs f(s/r) \geq \lfloor rs f(s/r) \rfloor 
\end{align*}
since 
\[
(rs/2) \delta_1 > \frac{1}{4} r^2 L \delta_1 \geq \frac{1}{4} C^2 L \delta_1 \geq 2 .
\]
Also, counting lattice points $(j,k)$ with $j=1$ shows that $\lfloor rs f(s/r) \rfloor = N(r,s)$ (lattice points with $j \geq 2$ cannot lie beneath the curve $r \Gamma(s)$ because $2s/r>L$). We conclude $N(r,s/2) > N(r,s)$, as we wished to show. 

An analogous argument proves that $s \notin S(r)$ when $s \in [(rM)^{-1},2(rM)^{-1})$, that is, when $s^{-1} \in ( \frac{1}{2} rM,rM]$. 
\end{proof}

\section{\bf Proof of \autoref{th:main1}}
\label{sec:mainproof}

We apply the three step method of Laugesen and Liu \cite{LL16}, which in turn was inspired by the method of Antunes and Freitas \cite{AF13} for the case where $\Gamma$ is a quarter circle.

First we estimate the remainder terms in \autoref{th:asymptotic}(b), which by the hypotheses of \autoref{th:main1} satisfy
\begin{align}
\big|N(r,s)-r^2 &\area (\Gamma)+r(s^{-1}L+sM)/2\big| \notag \\
\leq O&(r^{2/3}) + s^{-3/2} O(r^{1-2a_2}) +s^{3/2} O(r^{1-2b_2}) + (s^{-3/2}+s^{3/2}) O(r^{1/2}) \notag \\
&+ (s^2+s^{-2}) O(1) + s^{-1} O(r^{1-2a_1}) + s O(r^{1-2b_1}) + O(1) \label{eq:estimateO}
\end{align}
whenever $r \ge \max ( s/L, 1/sM)$. Here the implied constants depend only on the curve $\Gamma$ and not on $s$. 

Next we show $S(r)$ is bounded above and away from $0$. Applying \autoref{eq:estimateO} with $s=1$ gives that 
\[
r^2 \area(\Gamma)-cr/2 \leq N(r,1)
\]
for all large $r$, where the constant $c>0$ depends only on the curve $\Gamma$. Suppose $r$ is large enough that this estimate holds, and also that $r$ exceeds the constant $C$ in \autoref{lemma:bound2}.
Let $s \in S(r)$. Then $r \geq 2s/L$ by \autoref{lemma:bound2}, and so \autoref{prop:counting_positive} (which uses  convexity of the curve $\Gamma$) applies to give
\[
N(r,s) \leq {r^2}\area(\Gamma)-f(L/2)rs/2 .
\]
Naturally $N(r,1) \leq N(r,s)$, because $s \in S(r)$ is a maximizing value. Thus combining the preceding inequalities shows that $s \leq c/f(L/2)$, and so the set $S(r)$ is bounded above for all large $r$. Interchanging the roles of the horizontal and vertical axes, we similarly find $s^{-1}$ is bounded, and hence $S(r)$ is bounded away from $0$ for all large $r$. 

Lastly we show $S(r)$ approaches $\{1\}$ as $r \to \infty$. Let $s \in S(r)$, so that by above, $s$ and $s^{-1}$ are bounded above for all large $r$. Then the right side of \autoref{eq:estimateO} has the form $O(r^{1-2e})$, with the implied constant being independent of $s$; recall the exponent $e$ was defined in \autoref{th:main1}. Since $r \ge 2 \max ( s/L, 1/sM)$ by \autoref{lemma:bound2}, we see from \autoref{eq:estimateO} that 
\begin{align*}
N(r,s) & \leq r^2\area (\Gamma)-r(s^{-1}L+sM)/2 + O(r^{1-2e}) , \\
N(r,1) & \geq r^2\area (\Gamma)-r(L+M)/2 - O(r^{1-2e}) ,
\end{align*}
as $r \to \infty$. Using again that $N(r,1)\leq N(r,s)$, we deduce
\begin{equation} \label{eq:sinverse}
s^{-1}L+sM \leq L+M +O(r^{-2e}) .
\end{equation}
Taking $L=M$ gives $s^{-1}+s \leq 2 + O(r^{-2e})$, and so $s=1+O(r^{-e})$ by \autoref{le:squarecompletion} below, which proves the first claim in the theorem. For the second claim, when $s \in S(r)$ we have
\[
N(r,s) = r^2\area (\Gamma)-rL + O(r^{1-2e}) 
\]
by \autoref{eq:estimateO}, using also that $1 \leq (s+s^{-1})/2 \leq 1+O(r^{-2e})$ by \autoref{eq:sinverse} with $L=M$.

\begin{lemma}[An elementary comparison used above] \label{le:squarecompletion}
\[
s+s^{-1} \leq 2 +t \quad \Longrightarrow \quad | s - 1 | \leq 3\sqrt{t}
\]
whenever $s>0$ and $0<t<1$. 
\end{lemma}
\begin{proof} We have $( s^{1/2}-s^{-1/2})^2 = s + s^{-1} - 2 \leq t$, which implies $|s^{1/2}-s^{-1/2}| \leq t^{1/2}$. The number $1$ lies between $s^{1/2}$ and $s^{-1/2}$, and so $| s^{1/2} - 1 | \leq t^{1/2}$, which means $1 - t^{1/2} \leq s^{1/2} \leq 1 + t^{1/2}$. Now square both sides and use $t<t^{1/2}$ (since $t<1$). 
\end{proof}

\section{\bf Proof of \autoref{th:main3}}
\label{sec:main3proof}

First we need a two-term bound on the counting function in the closed first quadrant. Assume $f$ is convex and strictly decreasing on $[0,L]$, with $f(0)=M, f(L)=0$. Then we have the following analogue of \autoref{prop:counting_positive}. 

\begin{proposition}[Two-term lower bound on counting function]\label{prop:counting_nonneg} The number of nonnegative-integer lattice points lying inside $r\Gamma(s)$ in the closed first quadrant satisfies
\[
\mathcal{N}(r,s)\geq r^2\area(\Gamma) + \frac{1}{2}Mrs , \qquad r,s>0 .
\]
\end{proposition}
\begin{proof}
We need only prove the special case where $r=s=1$, because applying that case to the curve $r\Gamma(s)$ (which has vertical intercept $Mrs$) yields the general case of the proposition.

Clearly $\mathcal{N}(1,1)$ equals the total area of the squares of sidelength $1$ having lower left vertices at nonnegative integer lattice points inside the curve $\Gamma$. The union of these squares contains $\Gamma$, since the curve is decreasing.

\begin{figure}
\includegraphics[scale=0.4]{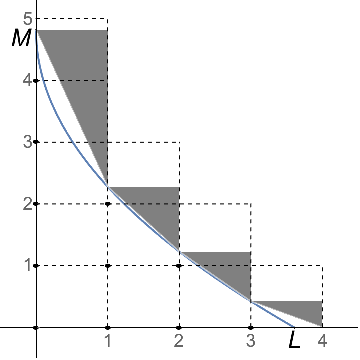}
\caption{\label{fig:counting_nonneg}Nonnegative integer lattice count $\mathcal{N}(1,1) \geq \area(\Gamma) + \area(\text{triangles})$, in proof of \autoref{prop:counting_nonneg}.}
\end{figure}

Consider the right triangles lying above chords of $\Gamma$, as shown in Figure~\ref{fig:counting_nonneg}. That is, for $i=1,\dots,\lfloor L \rfloor$ we take the triangle with vertices $(i-1,f(i-1)), (i,f(i)), (i,f(i-1))$, and the final triangle has vertices at $(\lfloor L \rfloor,f(\lfloor L \rfloor)), (\lceil L \rceil,0), (\lceil L \rceil,f(\lfloor L \rfloor))$.  

These triangles all lie above $\Gamma$, by concavity, and lie inside the collection of squares of sidelength $1$. Hence
\[
\mathcal{N}(1,1) \geq \area(\Gamma) + \area(\text{triangles}) = \area(\Gamma) + \frac{1}{2}M.
\]
\end{proof}

\subsection*{Proof of \autoref{th:main3}}
The number of lattice points lying on the axes and inside $r\Gamma(s)$ is 
\[
\lfloor Lr/s\rfloor+\lfloor Mrs\rfloor+1=Lr/s+Mrs+\rho(r,s) 
\]
where the error satisfies $|\rho(r,s)|\leq 1$. Thus $\mathcal{N}(r,s)$ and $N(r,s)$ (which, respectively, include and exclude the count of points on the axes) are connected by the formula
\[
\mathcal{N}(r,s)=N(r,s)+r(s^{-1}L +sM) +\rho(r,s) .
\]
Thus by estimate \autoref{eq:estimateO} from the proof of \autoref{th:main1} we have the asymptotic estimate
\begin{align}
\big|\mathcal{N}(r,s)&-r^2\area (\Gamma)-r(s^{-1}L+sM)/2\big| \notag \\
\leq O&(r^{2/3}) + s^{-3/2} O(r^{1-2a_2}) +s^{3/2} O(r^{1-2b_2}) + (s^{-3/2}+s^{3/2}) O(r^{1/2}) \notag \\
&+ (s^2+s^{-2}) O(1) + s^{-1} O(r^{1-2a_1}) + s O(r^{1-2b_1}) + O(1) \label{eq:asymptotic_nonneg}
\end{align}
whenever $r \ge \max ( s/L, 1/sM)$. 

Next we show that $\mathcal{S}(r)$ is bounded above and bounded below away from $0$.
Applying \autoref{eq:asymptotic_nonneg} with $s=1$ establishes that
\begin{equation} \label{eq:squarebound_nonneg}
r^2 \area(\Gamma)+cr/2 \geq \mathcal{N}(r,1)
\end{equation}
for all large $r$, where the constant $c>0$ depends only on the curve $\Gamma$. Suppose $r$ is large enough that this estimate holds. Let $s \in \mathcal{S}(r)$. Then \autoref{prop:counting_nonneg} applies to give
\[
\mathcal{N}(r,s) \geq {r^2}\area(\Gamma)+Mrs/2 .
\]

Since $s$ is a minimizer for the counting function $\mathcal{N}(r,\cdot)$ we must have $\mathcal{N}(r,1) \geq \mathcal{N}(r,s)$, and so the inequalities above imply that $s \leq c/M$. In other words, the set $\cS(r)$ is bounded above for all large $r$. Swapping the roles of the horizontal and vertical axes, we find by the same reasoning that $s^{-1}$ is bounded above, and hence the set $\cS(r)$ is bounded below away from $0$, for all large $r$.

Finally, we show $\mathcal{S}(r)$ approaches $\{1\}$ as $r \to \infty$. Let $s \in \mathcal{S}(r)$, so that $s$ and $s^{-1}$ are bounded above by Step~2, provided $r$ is large. Then the right side of estimate \autoref{eq:asymptotic_nonneg} is bounded by $O(r^{1-2e})$, with the implied constant being independent of $s$ and depending only on the curve $\Gamma$. From two applications of estimate \autoref{eq:asymptotic_nonneg} we deduce 
\begin{align*}
\mathcal{N}(r,s) & \geq r^2\area (\Gamma)+r(s^{-1}L+sM)/2 - O(r^{1-2e}) , \\
\mathcal{N}(r,1) & \leq r^2\area (\Gamma)+r(L+M)/2 + O(r^{1-2e}) ,
\end{align*}
as $r \to \infty$. Recalling that $\mathcal{N}(r,1) \geq \mathcal{N}(r,s)$ because $s$ is a minimizer, we deduce 
\[
s^{-1}L+sM \leq L+M +O(r^{-2e}) .
\]
Suppose $L=M$. Then $s=1+O(r^{-e})$ as $r \to \infty$ by \autoref{le:squarecompletion}. Also, estimate \autoref{eq:asymptotic_nonneg} implies for $s \in \mathcal{S}(r)$ that 
\[
\mathcal{N}(r,s) = r^2\area (\Gamma)+rL + O(r^{-2e}) , 
\]
where we used that $(s^{-1}+s)/2=1+O(r^{1-2e})$ by above.

\section*{Acknowledgments}
This research was supported by a grant from the Simons Foundation (\#429422 to Richard Laugesen), and by travel funding from the conference \emph{Shape Optimization and Isoperimetric and Functional Inequalities} at CIRM Luminy, France, November 2016.
Ariturk was supported by CAPES and IMPA of Brazil through the program P\'os-Doutorado de Excel\^encia.
We thank Shiya Liu for sharing the Mathematica files from which several figures in this paper were created.

\end{document}